\documentclass[a4paper]{amsart}
\usepackage{amsthm,amsmath,amssymb}
\usepackage[latin1]{inputenc}
 \newtheorem{thm}{Theorem}
 \newtheorem*{thm*}{Admissibility Theorem}
 \newtheorem{prop}[thm]{Proposition}
 \newtheorem{lemma}[thm]{Lemma}
 \newtheorem{cor}[thm]{Corollary}

 \theoremstyle{definition}

 \theoremstyle{remark}
 \newtheorem{remark}[thm]{Remark}

\def\Spec{{\rm Spec}\,}
\def\Spf{{\rm Spf}}
\def\Frac{{\rm Frac}~}
\def\Gr{{\rm Gr}}
\def\GL{{\rm GL}}
\def\GSp{{\rm GSp}}
\def\Ind{{\rm Ind}}
\def\rk{{\rm rk}}
\def\pr{{\rm pr}}
\def\ad{{\rm ad}}
\def\univ{{\rm univ}}
\def\dom{{\rm dom}}
\def\codim{{\rm codim}}

\def\defect{{\rm def}}
\def\dl{(\!(}
\def\dr{)\!)}

\begin{document}

\begin{title}
{Newton strata in the loop group of a reductive group}
\end{title}
\author{Eva Viehmann}
\address{Mathematisches Institut der Universit\"{a}t Bonn\\ Endenicher Allee 60\\53115 Bonn\\Germany}
\email{viehmann@math.uni-bonn.de}
\date{}

\begin{abstract}{We generalize purity of the Newton stratification to purity for a single break point of the Newton point in the context of local $G$-shtukas respectively of elements of the loop group of a reductive group. As an application we prove that elements of the loop group bounded by a given dominant coweight satisfy a generalization of Grothendieck's conjecture on deformations of $p$-divisible groups with given Newton polygons.}
\end{abstract}
\maketitle
\section{Introduction}\label{intro}
Let $\mathbb{F}_q$ be a finite extension of $\mathbb{F}_p$. All rings we consider are assumed to be $\mathbb{F}_q$-algebras. Let $G$ be a split connected reductive group over $\mathbb{F}_q$, and let $B\supseteq T$ be a Borel subgroup and a split maximal torus of $G$. We denote by $LG$ the loop group of $G$, i.e.~the ind-scheme over $\mathbb{F}_q$ representing the sheaf of groups for the fpqc-topology whose sections for an $\mathbb{F}_q$-algebra $R$ are given by $LG(R)=G(R\dl z\dr )$, see \cite{Faltings}, Definition 1. Let $K$ be the sub-group scheme of $LG$ with $K(R)=G(R[[z]])$. 

Let $k$ be a field (containing $\mathbb{F}_q$). We denote by $\sigma$ the Frobenius of $k$ over $\mathbb{F}_q$ and also of $k\dl z\dr $ over $\mathbb{F}_q\dl z\dr $. For algebraically closed $k$, the set $B(G)$ of $\sigma$-conjugacy classes $[b]=\{g^{-1}b\sigma(g)\mid g\in G(k\dl z\dr )\}$ of elements $b\in LG(k)$ is classified by two invariants, the Kottwitz point $\kappa_G(b)\in\pi_1(G)$ and the Newton point $\nu=\nu_b\in X_*(T)_{\mathbb{Q},\dom}$ (\cite{Kottwitz1}). Here $\pi_1(G)$ is the quotient of $X_*(T)$ by the coroot lattice, and $X_*(T)_{\mathbb{Q},\dom}$ is the set of elements of $X_*(T)\otimes_{\mathbb{Z}}\mathbb{Q}$ that are dominant with respect to the chosen Borel subgroup $B$. For the reformulation using $\pi_1(G)$ also compare \cite{RapoportRichartz}. We are interested in the variation of the $\sigma$-conjugacy class in families, i.e.~we want to consider an element $b\in LG(S)$ (or more generally a local $G$-shtuka on $S$) for an integral scheme $S$ and compare the $\sigma$-conjugacy classes in the various points of $S$. For the precise definition of the Newton point and the Kottwitz point in this case see Section \ref{secnot}. By \cite{HV1}, Proposition 3.4, the Kottwitz point $\kappa_G(b)$ is locally constant, hence constant on $S$. Therefore we view it as fixed, and from now on only consider the Newton point of $b$, as an element of the set of $\nu\in X_*(T)_{\mathbb{Q},\dom}$ for which there exists a $g\in LG(k)$ such that $[g]$ corresponds to $\nu$ and the fixed element of $\pi_1(G)$.

We equip $X_*(T)_{\mathbb{Q},\dom}$ with the usual partial ordering $\preceq$. By definition $\nu_1\preceq\nu_2$ if and only if the difference $\nu_2-\nu_1$ is a non-negative linear combination of positive coroots. For $\nu\in X_*(T)_{\mathbb{Q},\dom}$ and $b\in LG(S)$ let $\mathcal{N}_{\nu}\subseteq S$ be the locus where $b$ has Newton point $\nu$. By \cite{Grothendieck} resp.~\cite{RapoportRichartz} Theorem 3.6, \cite{HV1} Theorem 7.3 this defines a locally closed reduced subscheme $\mathcal{N}_{\nu}$ of $S$. Furthermore, the subscheme $\mathcal{N}_{\preceq\nu}=\bigcup_{\nu'\preceq\nu}\mathcal{N}_{\nu'}$ is closed for every $\nu$ (the so-called Grothendieck specialization theorem). The induced stratification of $S$ is called the Newton stratification associated with $b$.

An important property of the Newton stratification is its purity. A first purity theorem was shown by de Jong and Oort in \cite{JO}, Theorem 4.1 for the Newton stratification associated with an $F$-crystal. This result was later generalized by Vasiu \cite{Vasiu}, Main Theorem B in the context of $F$-crystals and by Zink \cite{Zink} for $p$-divisible groups. In \cite{HV1}, Theorem 7.4 it was noticed that the analog of Vasiu's purity theorem also holds in the function field case, i.e.~for Newton stratifications associated with elements of the loop group or with local $G$-shtukas. A generalization in a different direction was recently shown by Yang \cite{yang}, who proves an analog of de Jong and Oort's purity result, not for the Newton stratification itself but for the stratification associated with a single break point of the Newton polygon. Our first main result is a joint generalization of Yang's and Vasiu's purity results to the function field case. More precisely we show

\begin{thm}\label{thm1}
Let $S$ be an integral locally noetherian scheme and let $b\in LG(S)$. Let $j\in J(\nu)$ be a break point of the Newton point $\nu$ of $b$ at the generic point of $S$. Let $U_j$ be the open subscheme of $S$ defined by the condition that a point $x$ of $S$ lies in $U_j$ if and only if $\pr_{(j)}(\nu_{b}(x))=\pr_{(j)}(\nu)$. Then $U_j$ is an affine $S$-scheme.
\end{thm}
For a review of Newton points, their sets of break points $J(\nu)$ and the projections $\pr_{(j)}$ for general split reductive groups compare Section \ref{secnot}.

In \cite{Grothendieck} Grothendieck conjectured a converse to the Grothendieck specialization theorem: Let $k$ be a field of characteristic $p$. Let $\nu_1\preceq\nu_2$ be two Newton polygons of $p$-divisible groups and let $X_1$ be a $p$-divisible group over $k$ with Newton polygon $\nu_1$. Then there also exists a $p$-divisible group over the discrete valuation ring $k[[t]]$ such that the special fiber is $X_1$ and the Newton polygon of the generic fiber is $\nu_2$. For $p$-divisible groups with and without polarization this conjecture is shown by Oort in \cite{Oort1}. Our second main result yields in particular a generalization of this result to our context. 

For a coweight $\mu\in X_*(T)$ we denote by $z^{\mu}\in T(\mathbb{F}_q\dl z\dr )$ the image of $z$ under $\mu$ and by $\overline{\mu}$ the image of $\mu$ in $\pi_1(G)$. On the set of dominant coweights $X_*(T)_{\dom}$ we consider the partial ordering induced by the Bruhat order, i.e.~$\mu\preceq\mu'$ if and only if $\mu'-\mu$ is a non-negative integral linear combination of positive coroots.

\begin{thm}\label{thm2}
Let $\mu_1\preceq\mu_2\in X_*(T)$ be dominant coweights. Let $S_{\mu_1,\mu_2}=\bigcup_{\mu_1\preceq\mu'\preceq \mu_2} Kz^{\mu'}K$. Let $[b]$ be a $\sigma$-conjugacy class with $\kappa_G(b)=\overline{\mu}_1=\overline{\mu}_2$ as elements of $\pi_1(G)$ and with $\nu_b\preceq\mu_2$. Then the Newton stratum $\mathcal{N}_b=[b]\cap S_{\mu_1,\mu_2}$ is non-empty and pure of codimension $\langle \rho,\mu_2-\nu_b\rangle+\frac{1}{2}\defect(b)$ in $S_{\mu_1,\mu_2}$. The closure of $\mathcal{N}_b$ is the union of all $\mathcal{N}_{b'}$ for $[b']$ with $\kappa_G(b')=\overline{\mu_1}$ and $\nu_{b'}\preceq\nu_b$.
\end{thm}
Here $\rho$ is the half-sum of the positive roots of $G$ and the defect $\defect(b)$ is defined as $\rk ~G - \rk_{\mathbb{F}_q\dl z\dr}J_b$ where $J_b$ is the reductive group over $\mathbb{F}_q\dl z\dr$ with $J_b(k\dl z\dr)=\{g\in LG(\overline k)\mid gb=b\sigma(g)\}$ for every field $k$ containing $\mathbb{F}_q$ and with algebraic closure $\overline{k}$.
Note that in order to formulate the theorem, one needs to define the notions of codimension and of the closure of the infinite-dimensional schemes $\mathcal{N}_b$. The definitions are based on an admissibility result, cf. Remark \ref{remadm}.

It is unclear whether the approach used by Oort can be generalized to also prove Theorem \ref{thm2}. The idea of his proof is to construct first a $p$-divisible group $X_1'$ over $k[[t]]$ with special fiber $X_1$ and such that the $a$-invariant of the generic fiber is equal to 1. This first step uses as one important ingredient the purity theorem of de Jong and Oort. In a second step one studies in detail the universal deformation of the $p$-divisible group with $a$-invariant 1, and can in this case give a formula for the Newton polygon of deformations of the group. This then leads to a deformation of the $p$-divisible group with $a=1$ and Newton polygon $\nu_1$ over $k\dl t\dr $ to a $p$-divisible group with Newton polygon $\nu_2$. When trying to generalize this approach, a first non-trivial problem arises from the fact that for general groups $G$ instead of $\GL_n$ or $\GSp_{2n}$, and in particular for non-minuscule $\mu$ (even for the group $\GL_n$), there is up to now no suitable analog of the $a$-invariant, one of the key tools in Oort's proof. A second problem is that the proof uses the description of the two groups $\GL_n$ and $\GSp_{2n}$ by matrices for explicit calculations which do not generalize easily to general reductive groups. In the proof we propose here we thus proceed in a different way, and use various results on the Newton stratification on loop groups such as Theorem \ref{thm1} and the dimension formula for affine Deligne-Lusztig varieties (\cite{GHKR} and \cite{dimdlv}) together with Chai's results on lengths of chains of Newton points, \cite{Chai}. 

An immediate consequence of Theorem \ref{thm2} for $\mu_1=\mu_2$ is the following Corollary. In fact, in Section \ref{secgrothconj} we first prove the Corollary and then use it in the proof of Theorem \ref{thm2}. 

\begin{cor}\label{cormazur}
Let $b\in LG(k)$ for an algebraically closed field $k$. Let $\nu\in X_*(T)_{\mathbb{Q},\dom}$ be its Newton point and let $\kappa_G(b)\in \pi_1(G)$ be its Kottwitz point. Let $\mu\in X_*(T)_{\dom}$ with $\nu\preceq\mu$ and such that $\kappa_G(b)=\overline{\mu}$ in $\pi_1(G)$. Then $[b]\cap Kz^{\mu}K\neq\emptyset$.
\end{cor}

Corollary \ref{cormazur} is also referred to as the converse of Mazur's inequality for split groups. It gives a necessary and sufficient condition for non-emptiness of affine Deligne-Lusztig varieties in the affine Grassmannian. This Corollary has been previously shown using different methods by Kottwitz and Rapoport (\cite{KR}, for $\GL_n$ and $\GSp_{2n}$), Lucarelli (\cite{Lucarelli}, for classical groups) and then Gashi \cite{Gashi} in the general case. More precisely, these authors consider the Witt vector case instead of the function field case as we do. By \cite{GHKR2}, Proposition 13.3.1, the assertions of Corollary \ref{cormazur} for these two cases are equivalent.\\

\noindent{\it Acknowledgement.} I am grateful to Y.~Yang for making a preliminary version of her article on the purity theorem available to me. I thank the referee for pointing out some inaccuracies and for providing suggestions for improvements. This work was partially supported by the SFB/TR45 ``Periods, Moduli Spaces and Arithmetic of Algebraic Varieties'' of the DFG.

\section{Elements of the loop group and their Newton points}\label{secnot}

In this section we review the necessary results on the set of $\sigma$-conjugacy classes of elements in the loop group of a split reductive group, the main references being \cite{Kottwitz1} and \cite{Chai}.

As above let $G$ be a split reductive group over $\mathbb{F}_q$, and let $B\supseteq T$ be a Borel subgroup and a split maximal torus of $G$. Then an element $\eta$ of $X_*(T)$ or $X_*(T)_{\mathbb{Q}}=X_*(T)\otimes \mathbb{Q}$ is dominant if $\langle \alpha,\eta\rangle\geq 0$ for all positive roots $\alpha$ of $T$. We denote the subset of dominant elements by $X_*(T)_{\dom}$ resp.~$X_*(T)_{\mathbb{Q},\dom}$. Let $\widetilde{W}$ be the extended affine Weyl group of $G$. We have $\widetilde W\cong W\ltimes X_*(T)$ where $W$ denotes the finite Weyl group of $G$. We choose representatives of the elements of $\widetilde{W}$ in $G(\mathbb{F}_q\dl z\dr )$ and in this way view $\widetilde{W}$ as a subset of $G(\mathbb{F}_q\dl z\dr )$. 

The extended affine Weyl group $\widetilde{W}$ also has a decomposition $\widetilde W\cong \Omega\ltimes W_{\rm aff}$. Here $\Omega$ is the subset of elements of $\widetilde{W}$ which fix the chosen Iwahori subgroup $I=\{g\in K\mid g\in B\pmod{z}\}$ of $LG$. The second factor $W_{\rm aff}$ is the affine Weyl group of $G$. In terms of the decomposition $\widetilde{W}\cong W\ltimes X_*(T)$ it has the following description. Let $G_{sc}$ be the simply connected cover of $G$ and let $T_{sc}$ be the inverse image of $T$ in $G_{sc}$. Then $W_{\rm aff}\cong W\ltimes X_*(T_{sc})$ and $\Omega\cong X_*(T)/X_*(T_{sc})$. The affine Weyl group of $G$ is an infinite Coxeter group. It is generated by the simple reflections $s_i$ associated with the simple roots of $T$ in $G$ together with the simple affine root. We extend the corresponding length function $\ell$ to $\widetilde W$ by setting $\ell(x)=0$ for all $x\in \Omega$. The choice of $I$ also induces an ordering on $\widetilde{W}$, the Bruhat ordering. It is defined as follows. Let $x,y\in\widetilde W$ and let $x=\omega_xx'$ and $y=\omega_yy'$ be their decompositions into elements of $\Omega$ and $W_{\rm aff}$. Then $x\leq y$ if and only if $\omega_x=\omega_y$ and if there are reduced expressions $x'=s_{i_1}\dotsm s_{i_n}$ and $y'=s_{j_1}\dotsm s_{j_m}$ for $x'$ and $y'$ such that $(s_{i_1},\dotsc,s_{i_n})$ is a subsequence of $(s_{j_1},\dotsc,s_{j_m})$. On $X_*(T)_{\dom}$ it induces the ordering $\preceq$ described in the introduction. 

Let $k$ be a field containing $\mathbb{F}_q$. By the Bruhat decomposition we have $$LG(k)=\coprod_{\mu\in X_*(T)_{\dom}}K(k)z^{\mu}K(k).$$ If $b\in LG(k)$, the associated element $\mu_b$ with $b\in K(k)z^{\mu_b}K(k)$ is called the Hodge point of $b$. Note that it is invariant under extensions of $k$. Let $S$ be a reduced scheme. Then we associate with an element $b\in LG(S)$ the function $\mu_b$ mapping each closed point $x$ of $S$ to the Hodge point $\mu_b(x)=\mu_{b_x}$ of $b_x$. The element $b\in LG(S)$ is called bounded by some $\mu\in X_*(T)_{\dom}$ if $\mu_b(x)\preceq\mu$ for all $x\in S$. This definition is closely related to the notion of boundedness of the associated local $G$-shtukas as reviewed in Section \ref{secshtuka}. Note that in Section \ref{subsecadm} we use a slightly more general notion of boundedness for subsets of $LG$. The two definitions coincide for connected subsets resp. connected schemes $S$. 

One of Kottwitz's invariants of $\sigma$-conjugacy classes, the morphism $\kappa_G:G(L)\rightarrow \pi_1(G)$, can easily be described in terms of the Hodge point. Indeed, recall that $\pi_1(G)$ is the quotient of $X_*(T)$ by the coroot lattice. Then $\kappa_G(b)=[\mu_b]$ is the image of $\mu_b$ under the projection $X_*(T)\rightarrow \pi_1(G)$. The morphism $\kappa_G$ also induces a surjection $\kappa_G:\widetilde{W}\cong W\ltimes X_*(T)\rightarrow \pi_1(G)$. The subgroup $W_{\rm aff}$ of $\widetilde{W}$ is in the kernel of $\kappa_G$. On the subgroup $\Omega$ of the extended affine Weyl group, $\kappa_G$ induces the canonical isomorphism $\Omega\cong X_*(T)/X_*(T_{sc})\rightarrow \pi_1(G)$.

We further associate with $b\in LG(S)$ the function $\nu_b$ mapping each $x\in S(k)$ for some field $k$ to the Newton point $\nu_b(x):=\nu_{b_{\overline x}}$ of $b_{\overline x}$ where $\overline x$ is the associated point in $S(\overline k)$ for some algebraic closure $\overline k$ of $k$. This invariant is well defined, i.e.~independent of the choice of $\overline k$. For the purpose of this article it is enough to consider connected $S$. In this case the Kottwitz point $\kappa_G(b)$ is constant on $S$ and is assumed to be equal to some fixed element (see \cite{HV1}, Proposition 3.4). Then the Newton point determines the $\sigma$-conjugacy class of $b_x$ for every geometric point $x$. For a given Newton point $\nu$ let $\mathcal{N}_{\nu}\subseteq S$ be the Newton stratum associated with $\nu$ (and $S$, $b$) and let $\mathcal{N}_{\preceq\nu}=\bigcup_{\nu'\preceq\nu}\mathcal{N}_{\nu'}$ be the closed Newton stratum associated with $\nu$.

The set $\mathcal{N}(G)$ of elements of $X_*(T)_{\mathbb{Q},\dom}$ actually occurring as Newton points has been described by Kottwitz, and then further studied by Chai \cite{Chai}. We review some of Chai's results. Let $\alpha_1,\dotsc, \alpha_l$ be the simple roots of $T$ in $B$. With $\nu\in X_*(T)_{\mathbb{Q},\dom}$ we associate the set $J(\nu)=\{i\mid \langle \alpha_i,\nu\rangle>0\}.$ We call $J(\nu)$ the set of break points of $\nu$. This denomination is justified by the usual interpretation of Newton points for $\GL_n$ as polygons: For $G=\GL_n$ we have $X_*(T)_{\mathbb{Q}}\cong\mathbb{Q}^n$. Choosing $B$ to be the subgroup of lower triangular matrices, an element $\nu=(\nu_1,\dots,\nu_n)$ is dominant if $\nu_1\leq\dotsm\leq\nu_n$. We associate with $\nu$ the polygon that is the graph of the piecewise linear continuous function $[0,n]\rightarrow\mathbb{R}$ mapping $0$ to $0$ and with slope $\nu_i$ on $[i-1,i]$. Furthermore we choose $\alpha_i=(\dotsc,0,-1,1,0,\dotsc)$ to be the difference of the $(i+1)$st and the $i$th standard basis vector. Then $J(\nu)$ is the set of $i$ with $\nu_i<\nu_{i+1}$, which is the same as the set of first coordinates of the break points of the associated polygon.

Let $\omega_i^{\vee}\in X_*(T/Z)$ be the fundamental coweights (one for each $\alpha_i$) where $Z$ is the center of $G$. Let $\pr_{(i)}:X_*(T)_{\mathbb{Q}}\rightarrow \mathbb{Q}\cdot \omega_i^{\vee}$ be the orthogonal projection (compare \cite{Chai}, 6). On $\mathbb{Q}\cdot\omega_i^{\vee}$ we define the ordering $\preceq$ to be induced by the ordering on the coefficients in $\mathbb{Q}$.

Let $\nu',\nu\in \mathcal{N}(G)$ be two Newton points having the same image in $\pi_1(G)_{\mathbb{Q}}$. Then by \cite{Chai}, Lemma 6.2, $\nu'\preceq\nu$ if and only if $\pr_{(i)}(\nu')\preceq\pr_{(i)}(\nu)$ for all $i$. We assume that this is the case. Let $\mu\in X_*(T)_{\dom}$ with $\nu\preceq\mu$. We define
$$\delta(\nu',\nu)=\sum_i (\lceil\langle\omega_i,\mu-\nu'\rangle\rceil-\lceil\langle\omega_i,\mu-\nu\rangle\rceil)$$
where the $\omega_i$ denote the images of the fundamental weights of the derived group $G_{der}$ in $X^*(T)$. It is easy to see that this is independent of the choice of $\mu$. 

\begin{remark}\label{remchai}
\begin{enumerate}
\item\label{remchai2} Let $\nu'\preceq\nu$. By \cite{Chai}, Theorem 7.4 there is a chain of Newton points $\nu'=\nu_0\preceq \dotsm\preceq\nu_n=\nu$ with $n=\delta(\nu',\nu)$, and $\delta(\nu_i,\nu_{i+1})=1$ for all $i$. Each complete chain of Newton points between $\nu'$ and $\nu$ has length $\delta(\nu',\nu)$.
\item Let $\mu\in X_*(T)_{\dom}$ with $\nu\preceq\mu$, thus $\delta(\nu,\mu)=\sum_i \lceil\langle\omega_i,\mu-\nu\rangle\rceil$. By \cite{Kottwitz2} we have $$\delta(\nu,\mu)=\langle \rho,\mu-\nu\rangle+\frac{1}{2}\defect(b)$$ if $b\in LG(k)$ (over any field $k$) has Newton point $\nu$ and $\kappa_G(b)=\overline\mu$ and where the defect $\defect(b)$ is defined as in the introduction.
\end{enumerate}
\end{remark}

\begin{lemma}\label{lemchai}
Let $\nu\in \mathcal{N}(G)$ and let $j\in J(\nu)$. 
\begin{enumerate}
\item\label{1} There is a unique maximal $\nu^{j}\in \mathcal{N}(G)$ with $\nu^{j}\preceq \nu$ and $\pr_{(j)}(\nu^{j})\neq \pr_{(j)}(\nu)$. 
\item\label{2} $\pr_{(i)}(\nu^{j})=\pr_{(i)}(\nu)$ for all $i\in J(\nu)\setminus \{j\}$.
\item\label{3} $\delta(\nu^{j},\nu)=1$.
\end{enumerate}
\end{lemma}

\begin{proof}
By \cite{Chai}, proof of Lemma 6.2 (ii), the finite set $\{\eta\in\mathcal{N}(G)\mid \eta\preceq\nu, \pr_{(j)}(\eta)\neq\pr_{(j)}(\nu)\}$ is non-empty, and contains an element $\eta_j$ satisfying the equality in (\ref{2}). Let $\nu^j$ be the supremum of this finite set of Newton points, i.e.~the minimal element of $X_*(T)_{\mathbb{Q},\dom}$ with $\eta\preceq\nu^j$ for all $\eta$ as above, which Chai shows to exist and to be an element of $\mathcal{N}(G)$. It satisfies $\pr_{(j)}(\nu^j)=\max\{\pr_{(j)}(\eta)\}\neq\pr_{(j)}(\nu)$, hence (1). Furthermore, $\eta_j\preceq\nu^j\preceq\nu$ implies $\pr_{(i)}(\eta_j)\leq\pr_{(i)}(\nu^j)\leq\pr_{(i)}(\nu)$. Hence (2) follows for $\nu^j$. For (\ref{3}) we use that if $\nu^j\preceq\nu'\preceq\nu$ with $\nu'\neq\nu$ then by the second equality of \cite{Chai}, Lemma 6.2 (i) there is a $j'\in J(\nu)$ with $\pr_{(j')}(\nu')\neq\pr_{(j')}(\nu)$. Thus the claim follows from \eqref{1}, \eqref{2}, and Remark \ref{remchai} (1).
\end{proof}
\begin{remark}\label{remnew}
Now assume that $\nu'\preceq\nu$ with $\delta(\nu',\nu)=1$. Then there is a $j\in J(\nu)$ with $\pr_{(j)}(\nu')\neq\pr_{(j)}(\nu)$ (\cite{Chai}, Lemma 6.2 (i)). As the supremum of $\nu'$ and $\nu^j$ satisfies the defining conditions of $\nu^j$, and as $\delta(\nu',\nu)=1$ implies that there is no element of $\mathcal{N}(G)$ strictly lying between $\nu'$ and $\nu$ we see that $\nu'=\nu^j$.
\end{remark}

\section{Review of local $G$-shtukas}\label{secshtuka}

We recall some definitions and results related to local $G$-shtukas and their deformations, the main reference being \cite{HV1}.

Let $S$ be an $\mathbb{F}_q$-scheme. A local $G$-shtuka $\underline{\mathcal{G}}$ over $S$ is a pair $\underline{\mathcal{G}}=(\mathcal{G},\phi)$ where $\mathcal{G}$ is a $K$-torsor on $S$ for the \'etale topology and where $\phi:\sigma^*(\mathcal{LG})\rightarrow \mathcal{LG}$ is an isomorphism of $LG$-torsors. Here $\mathcal{LG}$ is the $LG$-torsor associated with $\mathcal{G}$.    

We assume for the moment that the torsor $\mathcal{G}$ is trivial and choose a trivialisation $\mathcal{G}\cong K_S$. Note that this is in particular the case if $S$ is the spectrum of an algebraically closed field or more generally of a strictly henselian ring. The trivialization induces an identification of the isomorphism $\phi$ with $b\sigma$ for some element $b$ of $LG(S)$. Changing the trivialization replaces $b$ by a $\sigma$-conjugate under $K(S)$. Let now $\underline{\mathcal{G}}$ be a not necessarily trivializable local $G$-shtuka over $S$. Then we can associate with each point of $S$ a Newton point (or rather a $\sigma$-conjugacy class) and a Hodge point of $\underline{\mathcal{G}}$ which are defined as the corresponding invariants of a trivialisation of the local $G$-shtuka over an algebraic closure of the field of definition of the given point. The two invariants are independent of the choice of the algebraic closure and of the chosen trivialization of $\underline{\mathcal{G}}$. We call the functions mapping the closed points of $S$ to the corresponding invariants of $\underline{\mathcal{G}}$ in that point the Newton point resp. the Hodge point of $\underline{\mathcal{G}}$. In this way we also obtain the notion of boundedness of $\underline{\mathcal{G}}$ by some $\mu\in X_*(T)_{\dom}$ as long as $S$ is reduced. To define bounded local $G$-shtukas over a non-reduced scheme one has to be more careful (compare \cite{HV1}, Example 3.12). Let $\psi:G\rightarrow \GL(V_{\psi})$ be a representation of $G$. Then let $\underline{\mathcal G}_{\psi}$ be the fpqc-sheaf $\mathcal{G}_{\psi}$ of $\mathcal{O}_S[[ z]]$-modules on $S$ associated with the presheaf $$Y\mapsto\left(\mathcal{G}(Y)\times(V_{\psi}\otimes_{\mathbb{F}_q}\mathcal{O}_S[[ z]](Y))\right)/K(Y)$$ (compare \cite{HV1}, Section 3) together with the morphism $\phi_{\psi}:\sigma^*(\mathcal{G}_{\psi}\otimes_{\mathcal{O}_S[[z]]}\mathcal{O}_S\dl z\dr)\rightarrow\mathcal{G}_{\psi}\otimes_{\mathcal{O}_S[[z]]}\mathcal{O}_S\dl z\dr$ induced by $\phi$. For a dominant weight $\lambda$ of $G$ let $V(\lambda)=(\Ind_{\overline{B}}^G((-\lambda)_{\dom}))^{\vee}$ be the associated Weyl module. Here $\overline{B}$ is the Borel subgroup of $G$ opposite of $B$ and $(-\lambda)_\dom$ denotes the dominant element in the Weyl group orbit of $-\lambda$. Then $\underline{\mathcal{G}}$ is bounded by $\mu\in X_*(T)_{\dom}$ if for every dominant weight $\lambda$
\begin{align*} 
\phi_\lambda(\mathcal {G}_{\lambda})&\subseteq z^{-\langle(-\lambda)_\dom,\mu\rangle}\mathcal{G}_{\lambda}\subseteq\mathcal{G}_{\lambda}\otimes_{\mathcal{O}_S[[z]]}\mathcal{O}_S\dl z\dr\quad\text{and}\\
\overline\mu&=\overline{\mu_{\underline{\mathcal{G}}}(x)} \text{ in }\pi_1(G)\text{ for all }x\in S.
\end{align*}

Let $\underline{\mathcal{G}}_0\cong(K_k,b\sigma)$ be a local $G$-shtuka over an algebraically closed field $k$ which is bounded by $\mu\in X_*(T)_{\dom}$. The formal deformation functor of $\underline{\mathcal G_0}$ for deformations bounded by $\mu$ is the functor assigning to each Artinian local $k$-algebra $A$ with residue field $k$ the set of isomorphism classes of pairs $(\underline{\mathcal G},\beta)$. Here $\underline{\mathcal G}$ is a local $G$-shtuka over $\Spec A$ bounded by $\mu$ and $\beta:\underline{\mathcal G_0}\rightarrow\underline{\mathcal G}\otimes_A k$ is an isomorphism of local $G$-shtukas. Two pairs $(\underline{\mathcal G},\beta)$ and $(\underline{\mathcal G}',\beta')$ are isomorphic if there exists an isomorphism $\eta:\underline{\mathcal G}\rightarrow\underline{\mathcal G}'$ with $\beta'=(\eta\otimes_A k)\circ\beta$. 

Let $\Gr=LG/K$ be the affine Grassmannian of $G$. By the Bruhat decomposition we have $\Gr=\coprod_{\mu\in X_*(T)_{\dom}}Kz^{\mu}K/K$. We associate with $\mu\in X_*(T)_{\dom}$ the closed subscheme $\Gr_{\preceq\mu}=\coprod_{\mu'\preceq\mu}Kz^{\mu'}K/K$ of $\Gr$. By \cite{HV1}, Theorem 5.6, the formal deformation functor of $\underline{\mathcal G_0}$ for deformations bounded by $\mu$ is pro-represented by the following variant: Let $\Gr'=K\backslash LG$ and $\Gr'_{\preceq\mu}=\coprod_{\mu'\preceq(-\mu)_{\dom}}K\backslash Kz^{\mu'}K$. Then the formal deformation functor is pro-represented by the completion $R_{\preceq\mu}$ of $\Gr'_{\preceq\mu}$ in $b^{-1}$. Especially, the reduced quotient of the universal deformation ring is a complete normal noetherian integral domain of dimension $\langle 2\rho, \mu\rangle$ and Cohen-Macaulay (\cite{HV1}, Prop. 5.9).

By \cite{HV1}, Proposition 3.16, we have a bijection between local $G$-shtukas bounded by $\mu$ over $\Spec R_{\preceq\mu}$ and over $\Spf~R_{\preceq\mu}$. Especially, the universal local $G$-shtuka $\underline{\mathcal{G}}_{\univ}$ over $\Spf~R_{\preceq\mu}$ induces a local $G$-shtuka and hence a Newton stratification on $\Spec R_{\preceq\mu}$.

We now consider the case $G=\GL_n$. Let $\mathcal{G}$ be a $K$-torsor on an $\mathbb{F}_q$-scheme $S$. By \cite{HV1}, Proposition 2.3, $\mathcal{G}$ is Zariski-locally on $S$ isomorphic to the trivial $K$-torsor. Furthermore we have a bijection between $K$-torsors $\mathcal{G}$ on $S$ and sheaves of $\mathcal{O}_S[[ z]]$-modules on $S$ which locally for the Zariski-topology on $S$ are isomorphic to $\mathcal{O}_S[[ z]]^{n}$. It is given by mapping $\mathcal{G}$ to $\mathcal{G}_{\psi}$ where $\psi$ is the standard representation of $\GL_n$ (\cite{HV1}, Lemma 4.2). If $\underline{\mathcal{G}}=(\mathcal{G},\phi)$ is a local $\GL_n$-shtuka on $S$, the morphism $\phi$ induces a corresponding morphism $\phi:\sigma^*(\mathcal{G}_{\psi}\otimes_{\mathcal{O}_S[[z]]}\mathcal{O}_S\dl z\dr)\rightarrow \mathcal{G}_{\psi}\otimes_{\mathcal{O}_S[[z]]}\mathcal{O}_S\dl z\dr$. The pair $(\mathcal{G}_{\psi},\phi)$ is called the local shtuka of rank $n$ associated with $\underline{\mathcal{G}}$. As the functor assigning to each local $\GL_n$-shtuka its corresponding local shtuka of rank $n$ is a bijection, we sometimes identify $\underline{\mathcal{G}}$ and $(\mathcal{G}_{\psi},\phi)$.

\section{Purity}\label{secpurity}
The main goal of this section is to prove Theorem \ref{thm1}. We also discuss some implications of the theorem.

\subsection{Reduction of the purity theorem to a special case}

For the proof of Theorem \ref{thm1} we generalize the proof of \cite{Vasiu}, Theorem 6.1. The main strategy is the same as in Vasiu's proof, however, the details are changed considerably due to the weaker assumptions we make. We first reduce Theorem \ref{thm1} to the following special case.
\begin{thm}\label{thm1'}
Assume in the situation of Theorem \ref{thm1} that $G=\GL_n$, and that $S$ is a noetherian integral, normal, excellent, affine scheme. Using the notation for Newton polygons (for $\GL_n$) explained in Section \ref{secnot} let $j$ be the first break point of the generic Newton polygon of $b$ and assume that $j=(1,0)$, i.e.~that all Newton slopes are non-negative and slope $0$ occurs precisely once. Let $U_j$ be the open subscheme of $S$ defined by the condition that a point $x$ of $S$ lies in $U_j$ if and only if $j$ lies on the Newton polygon of $b_x$. Then $U_j$ is an affine $S$-scheme.
\end{thm}
\begin{proof}[Proof of Theorem \ref{thm1} using Theorem \ref{thm1'}]
We begin by justifying the assumption $G=\GL_n$. Assume that Theorem \ref{thm1} is true for $G=\GL_n$. This implies the theorem for all $G$ which are a product of finitely many factors $\GL_{r_i}$. For general $G$ one considers a finite number of representations $\lambda_i:G\rightarrow \GL(V_{\lambda_i})$ for $i=1,\dotsc,l$ distinguishing the different Newton polygons on $S$ (for example one for each simple factor of the adjoint group $G_{\ad}$). Then the Newton stratifications of $S$ corresponding to  $b$ and  $(\lambda_1(b),\dotsc,\lambda_l(b))$ coincide.

The assumptions on $S$ made in Theorem \ref{thm1'} can be justified in the same way as in Vasiu's proof of the purity theorem \cite{Vasiu}, Proof of Theorem 6.1, p. 291.

It remains to prove that we may assume that $j$ is the first break point and equal to $(1,0)$. Let $(\Lambda,b\sigma)$ be the local shtuka of rank $r$ corresponding to $b$ whose Newton stratification we want to study. Let $(m,m')\in \mathbb{Z}^2$ be the first break point of its Newton polygon. Replacing $\Lambda$ by its $m$th exterior power $\wedge^m(\Lambda)$ we obtain a local shtuka whose Newton slopes in the various points of $S$ are of the form $\nu_{i_1}+\dotsm+\nu_{i_m}$ where $1\leq i_1<\dotsm<i_m\leq r$ are pairwise distinct and where the $\nu_i$ are the Newton slopes of $b$ in that point. In particular, the Newton stratifications induced by $\Lambda$ and $\wedge^m(\Lambda)$ coincide, and $(1,m')$ is the first break point of the Newton polygon of $\wedge^m(\Lambda)$. Hence we may assume that $m=1$. Multiplying $b$ by the central element $z^{-m'}\in LG(\mathbb{F}_q)$ we may assume that $m'=0$. 
\end{proof}

\subsection{Slope filtrations up to isogeny}\label{subsec4.2}
In this subsection we assume that $G=\GL_n$, that $T$ is the diagonal torus and that $B$ is the Borel subgroup of lower triangular matrices. We denote by $P$ the maximal standard parabolic subgroup of $G$ corresponding to the simple root $e_n-e_{n-1}$, by $M$ its Levi component containing $T$, and by $N$ its unipotent radical. Explicitly, we have that $M\cong \GL_{n-1}\times \mathbb{G}_m$ consists of diagonal block matrices and $P=MB$. We denote by $LP$ and $LN$ the loop groups of $P$ respectively of $N$ which are defined in the same way as for reductive groups.

\begin{lemma}\label{lemzink}
Let $\underline{\mathcal{G}}$ be a local shtuka of rank $n$ over a field $k$ which is bounded by some $\mu\in X_*(T)_{\dom}$ and whose Newton slopes are all non-negative. Then there is an isogeny $\rho:\underline{\mathcal{G}}\rightarrow \underline{\mathcal{H}}$ such that the rank of its cokernel is bounded by a bound $\mu'$ which only depends on $\mu$ and such that $\underline{\mathcal{H}}$ has only non-negative Hodge slopes.
\end{lemma}
This is essentially the function field analog of \cite{Zink2}, Lemma 9. The construction of $x$ and $\mu'$ below is done in an explicit way because we do not want to use Corollary \ref{cormazur}. Also note that we do not assume $k$ to be algebraically closed. 
\begin{proof}
The local shtuka $\underline{\mathcal{G}}$ is isomorphic to $(k[[z]]^n,b\sigma)$ for some $b\in L\GL_n(k)$. We claim that there is an element $x\in [b]\cap \widetilde{W}$ whose translation part has only non-negative slopes. To construct $x$ let $\nu$ be the Newton point of $b$. We write the break points and the end point of the associated polygon in the form $(n_1+\dotsm+n_i,m_1+\dotsm+m_i)$ for $i=1,\dotsc,j$ and $m_i,n_i\in \mathbb{N}$ with $n_i>0$ and $n_1+\dotsm+n_j=n$. The slope of the $i$th part of the Newton polygon is equal to $m_i/n_i$. Let $M_{\nu}$ be the centralizer of the Newton point of $x$, the standard Levi subgroup of $\GL_n$ isomorphic to $\GL_{n_1}\times\dotsm\times \GL_{n_j}$. For each simple factor $\GL_{n_i}$ let $x_i=\pi(n_i)^{m_i}\in \widetilde{W}$ where $\pi(n_i)$ maps the $l$th basis vector $e_l$ to the $(l+1)$st for $l<n_i$ and to $ze_1$ for $l=n_i$. A direct calculation shows that the slopes of the translation part of $x_i$ are all in $\{\lfloor m_i/n_i\rfloor,\lceil m_i/n_i\rceil\}$. Furthermore, $x_i$ is basic (in $\GL_{n_i}$) of slope $m_i/n_i$. Let $x=(x_1,\dotsc,x_j)\in \widetilde{W}$. Then the Newton points of $b$ and $x$ coincide, hence $x\in [b]$. Let $\mu'\in X_*(T)$ be the dominant element conjugate to the translation part of $x$. As all Newton slopes $m_i/n_i$ of $b$ are non-negative, the same holds for the slopes of $\mu'$. 

By construction there is a $g\in L\GL_n(\overline{k})$ such that $g^{-1}b\sigma(g)=x\in K(k)z^{\mu'}K(k)$. This implies that $\Lambda=g^{-1}(\overline{k}[[z]]^n)$ satisfies  
\begin{equation}\label{eqzink1}
b\sigma(\Lambda)\subseteq \Lambda. 
\end{equation}
By replacing $\Lambda$ by $z^c\Lambda$ for a sufficiently small $c\in\mathbb{Z}$ we may in addition assume that 
\begin{equation}\label{eqzink2}
\overline{k}[[z]]^n\subseteq\Lambda.
\end{equation}
Finally we may choose a minimal $\Lambda$ among those satisfying \eqref{eqzink1} and \eqref{eqzink2}. Let $\Lambda_0=k[[z]]^n$ and let $\Lambda'=\Lambda_0+b\sigma(\Lambda_0)+\dotsm+(b\sigma)^{n-1}(\Lambda_0)\subseteq \Lambda$. Then the same (Nakayama-)argument as in \cite{Zink2}, proof of Lemma 9 shows that $b\sigma(\Lambda')\subseteq \Lambda'$, hence all Hodge slopes of $(\Lambda',b\sigma)$ are non-negative. Furthermore $\Lambda'$ is defined over $k$. As $b\sigma$ is bounded by $\mu$, the $(b\sigma)^i$ are bounded by $i\cdot \mu$ (where the multiplication is componentwise). Hence the rank of $\Lambda'/\Lambda_0$ is bounded by a bound which only depends on $\mu$. This induces the desired isogeny between $\underline{\mathcal{G}}=(\Lambda_0,b\sigma)$ and $\underline{\mathcal{H}}=(\Lambda',b\sigma)$.
\end{proof}

\begin{lemma}\label{lem2}
Let $\mu\in X_*(T)$ be dominant. Then there exists a constant $l=l(\mu)$ with the following property. Let $R$ be a discrete valuation ring and let $b\in L\GL_n(R)$ be bounded by $\mu$ and such that the Newton slopes of $b$ are non-negative in the generic and in the special point of $\Spec R$. Then there is a finite \'etale extension $R'$ of $R$ and $g\in L\GL_n(R')$ bounded by some $\mu'=(\mu'_1,\dotsc,\mu'_{n})$ with $0\leq \mu'_i\leq l$ for all $i$ and such that $g^{-1}b\sigma(g)\in L\GL_n(R')$ has only non-negative Hodge slopes in every point of $\Spec R'$.
\end{lemma}
\begin{proof}
Let $k=\Frac R$. Then by Lemma \ref{lemzink} we obtain a $\mu'$ and a $g_{k}\in L\GL_n(k)$ with the desired properties over $\Spec k$. Apparently the same holds for every element of $g_{k}K(k)$. The image of $g_{k}$ in the affine Grassmannian $\Gr=L\GL_n/K$ lies in the closed subscheme $\Gr_{\preceq\mu'}$. Note that $\Gr_{\preceq\mu'}$ is projective over $\mathbb{F}_q$. As $R$ is a discrete valuation ring, there is a $\overline{g}\in \Gr_{\preceq\mu'}(R)$ with generic point $g_{k}K_{k}$. By \cite{HV2} Lemma 2.3, $\overline{g}$ can be lifted to an element $g\in L\GL_n(R')$ over some finite \'etale extension $R'$ of $R$. This element generically satisfies the claimed conditions, and hence it satisfies them on all of $\Spec R'$.
\end{proof}

The following lemma is one particular instance of the Hodge-Newton filtration. Although this filtration is known to exist in various contexts we could not find a reference for this precise statement, so we include a proof here.
\begin{lemma}\label{prop3}
Let $R$ be a discrete valuation ring and let $\underline{\mathcal{G}}=(R[[z]]^n,b\sigma)$ be a local shtuka of rank $n$ on $\Spec R$. Assume that $\underline{\mathcal{G}}$ is effective, i.e.~its Hodge slopes are non-negative in all points of $\Spec R$. Assume furthermore that $(1,0)$ is a break point of the Newton polygon of $\underline{\mathcal{G}}$ in all points of $\Spec R$. Then there is a filtration $\underline{\mathcal{G}_0}\hookrightarrow\underline{\mathcal{G}}$ where $\underline{\mathcal{G}_0}=(R[[z]],b_0\sigma)$ for some $b_0$ of (Newton and Hodge) slope $0$, and where $\underline{\mathcal{G}}/\underline{\mathcal{G}_0}$ is a local shtuka of rank $n-1$ having strictly positive Newton slopes. 
\end{lemma}

\begin{proof}
We follow the proof of \cite{Katz}, Theorem 2.4.2. Let $b=(b_{ij})\in {\rm Mat}_{n\times n}(R[[z]])$ be the matrix corresponding to $b$. Considering the image of $b$ modulo $z$ and in the special point of $\Spec R$, we see that the $b_{ij}$ generate the unit ideal in $R[[z]]$. The same is true for the matrix coefficients of every power $(b\sigma)^c$. The Newton slope $0$ occurs with multiplicity 1. Thus for every $r$ there is a $c_r$ such that for every $c\geq c_r$ all Hodge slopes of $\wedge^2(b\sigma)^c$ are $\geq r$ at each point of $\Spec R$. Hence all $2\times 2$-minors of $(b\sigma)^c$ are in $z^rR[[z]]$. Let $\underline{\mathcal{G}_0}$ be the intersection of the images of the $(b\sigma)^c$, then $\underline{\mathcal{G}_0}\cong(R[[z]],b_0\sigma)$ for some $b_0$ and $\mathcal{G}/\mathcal{G}_0$ is a free $R[[z]]$-module of rank $n-1$. The slope assertions follow as they can be checked pointwise. 
\end{proof}

\begin{remark}
From the lemma one can easily deduce the following statement which is the general form of the Hodge-Newton filtration for local shtukas over discrete valuation rings, but which we do not need in this article. Let $R$ be a discrete valuation ring and let $\underline{\mathcal{G}}=(R[[z]]^n,b\sigma)$ be a local shtuka on $\Spec R$ of rank $n$. Assume that there is a break point $(a,m)$ of the Newton polygon of $\underline{\mathcal{G}}$ which also lies on the Hodge polygon in all points of $\Spec R$. Then there is a filtration $\underline{\mathcal{G}_0}\hookrightarrow\underline{\mathcal{G}}$ associated with $(a,m)$. Here $\underline{\mathcal{G}_0}\cong(R[[z]]^a,b_0\sigma)$ is such that its Hodge and Newton polygon consist of the first $a$ components of the corresponding polygons for $\underline{\mathcal{G}}$ and $\underline{\mathcal{G}}/\underline{\mathcal{G}_0}$ is a local shtuka of rank $n-a$ whose associated polygons consist of the remaining $n-a$ components. Indeed, this follows by a standard reduction argument, applying the lemma to $(\wedge^a(R[[z]]^n),z^{-m}\wedge^a(b\sigma))$, see \cite{Katz}, proof of Theorem 2.4.2.
\end{remark}

\begin{prop}\label{corslopefilt}
Let $\mu\in X_*(T)_{\dom}$. Then there exists a constant $l=l(\mu)$ with the following property. Let $d\in\mathbb{N}$. Let $R$ be a discrete valuation ring and let $b\in L\GL_n(R)$ be bounded by $\mu$. We assume that $(1,0)$ is a break point of the Newton polygon of $b$ in each point of $\Spec R$. Then there is a finite \'etale extension $R'$ of $R$ and a $g\in L\GL_n(R')$ bounded by some $\mu'=(\mu'_1,\dotsc,\mu'_{n})$ with $0\leq \mu'_i\leq l$ for all $i$ and such that $b'=g^{-1}b\sigma(g)\in LP(R')$ has only integral coordinates and that the image $b'_0$ in  the quotient $L\mathbb{G}_m(R')$ of $(L\GL_{n-1}\times L\mathbb{G}_m)(R')\cong LM(R')$ satisfies $b'_0\equiv 1\pmod{z^dR'[[z]]}$.

Furthermore, there is a finite purely inseparable extension $\widetilde{R}$ of $R'$ and an element $h\in (K\cap LN)(\widetilde{R})$ such that $h^{-1}b'\sigma(h)=b_Mb_N$ for $b_M\in LM(\widetilde{R})$ and $b_N\in (K\cap LN)(\widetilde{R})$ with $b_N\equiv 1\pmod{z^d}$.
\end{prop}

\begin{proof}
From Lemma \ref{lem2} and Lemma \ref{prop3} we obtain the weaker statement where $b'_0\in\mathbb{G}_m(R'[[z]])$ with $v_z(b'_0)=0$ in each point of $\Spec R'$. To replace $b'_0$ by an element as above we need some $g\in \mathbb{G}_m(R''[[z]])$ for some finite \'etale extension $R''$ of $R'$ with $g^{-1}b'_0\sigma(g)\equiv 1\pmod{z^dR''[[z]]}$. This is equivalent to $g\sigma(g)^{-1}\equiv b'_0\pmod{z^dR''[[z]]}$, a relation which can indeed be solved over some finite \'etale extension $R''$ of $R'$.

To show the second statement we decompose $b'$ as $b'_Mb'_N$ with $b'_M\in LM(R')$ and $b'_N\in LN(R')$. From the first assertion we obtain that $b'_N\in LN(R')\cap K(R')$. After passing to a first finite purely inseparable extension of $R'$ we may assume that $b'_N=\delta_1y_1$ for some $\delta_1\in LN(R')\cap K(R')$ with $\delta_1\equiv 1\pmod{z^d}$ and with $y_1=\sigma(y_1')$ for some $y_1'\in LN(R')\cap K(R')$. We $\sigma$-conjugate $b'$ by $(y_1')^{-1}$ to obtain $y_1'b'_M\delta_1=b'_M((b'_M)^{-1}y_1'b'_M\delta_1)$. Here our assumption that the coordinates of $b'_M$ are integral and that $b'_0\equiv 1$ implies that $(b'_M)^{-1}y_1'b'_M\delta_1\in LN(R')\cap K(R')$. We iterate this process. If we write $b'_M=(b'_{n-1},b'_{0})\in (L\GL_{n-1}\times L\mathbb{G}_m)(R')\cong LM(R')$ we have that the Newton slopes of $b'_{n-1}$ are all strictly positive, hence greater than the Newton slope 0 of $b'_0$. This implies that the morphism $LN(k)\cap K(k)\rightarrow LN(k)\cap K(k)$ with $y\mapsto (b'_M)^{-1}\sigma^{-1}(y)b'_M$ is topologically nilpotent for every perfect field $k$. Thus after finitely many steps we have replaced $b'_Mb'_N$ by an element of $b'_M\cdot (LN(R')\cap K(R'))$ whose second factor is congruent to $1$ modulo $z^d$ (a condition that can be checked on geometric points).
\end{proof}

Similar results for $p$-divisible groups by Zink resp. Oort and Zink can be found in \cite{Zink2} and \cite{OortZink}.

\subsection{Admissibility}\label{subsecadm}

In this subsection we allow $G$ to be any split connected reductive group over $\mathbb{F}_q$. We fix some $c\in\mathbb{N}$. Let $K_c$ be the sub-group scheme of $K$ whose $R$-valued points are given by $\{g\in K(R)\mid g\equiv 1\pmod{z^{c}}\}$. We call a subset of $LG(k)$ bounded if it is contained in a finite union of $K$-double cosets. 

Let us recall the following admissibility result.
\begin{thm*}[\cite{HV1}, Theorem 10.1]
Let $k$ be algebraically closed and let $\mathcal{B}$ be a bounded subset of $LG(k)$. Then there is a $c\in \mathbb{N}$ such that for each $d\in\mathbb {N}$, each $b\in \mathcal{B}$ and each $h\in K_{c+d}(k)$ there is an $\alpha\in K_d(k)$ with $bh=\alpha^{-1}b\sigma(\alpha)$. 
\end{thm*}
\begin{cor}\label{coradmissible}
Let $k$ be algebraically closed. Let $n>0$ and $\mu\in X_*(T)_{\dom}$. Then there is a $c>0$ with the following property. Let $\underline{\mathcal{G}}=(k[[z]]^n,b\sigma)$ be a local shtuka bounded by $\mu$. Let $g:k[[z]]\rightarrow k[[z]]^{n}$ be such that $b\circ\sigma(g)\in K_{\GL_n,d}(k)gK_{\GL_1,d}(k)$ for some $d>c$. Then there is an element of $K_{\GL_n,d-c}\,g$ inducing a morphism of local shtukas between $(k[[z]],1\cdot\sigma)$ and $\underline{\mathcal{G}}=(k[[z]]^n,b\sigma)$.
\end{cor}
\begin{proof}
The assumption on $g$ implies that $b\circ\sigma(g)=hg$ for some $h\in K_{\GL_n,d}(k)$. As $b$ is bounded by $\mu$ there is a constant $e$ only depending on $\mu$ such that $h^{-1}b=bh'$ for some $h'\in K_{\GL_n,d-e}(k)$ if $d>e$ (\cite{HV2}, Lemma 6.3). Applying the theorem to $bh'$ we obtain the corollary.
\end{proof}

\begin{remark}\label{remcoradmissible}
Similarly we obtain the following statement. Let $k$ be algebraically closed. Let $n>0$ and $\mu\in X_*(T)_{\dom}$. Then there is a $c>0$ with the following property. Let $\underline{\mathcal{G}}=(k[[z]]^{n},b\sigma)$ be a local shtuka bounded by $\mu$. Let $g:k[[z]]^{n}\rightarrow k[[z]]$ be such that $gb\in K_{\GL_1,d}(k)\sigma(g)K_{\GL_n,d}(k)$ for some $d>c$. Then there is an element of $gK_{\GL_{n},d-c}$ inducing a morphism between $\underline{\mathcal{G}}$ and $(k[[z]],1\cdot\sigma)$.
\end{remark}

\begin{remark}\label{remadm}
Let $\mathcal{B}=\bigcup_{\mu_1\preceq\mu'\preceq\mu_2}Kz^{\mu'}K$ for some fixed dominant $\mu_1\preceq\mu_2\in X_*(T)$ be the subscheme of $LG$ we are interested in for Theorem \ref{thm2}. Then by the Admissibility Theorem we obtain a $c>0$ such that the isomorphism class, and thus in particular the $\sigma$-conjugacy class of each element $b\in \mathcal{B}$ is well-defined by its image in $LG/K_c$. Thus the Newton strata in $\mathcal{B}$ are the inverse images of certain locally closed subschemes of $\mathcal{B}/K_c$, a finite-dimensional scheme. In particular, the notions of codimension and closure of a stratum are well-defined (and independent of the choice of $c$, provided that it is sufficiently large).
\end{remark}

A simpler local variant of admissibility is the following assertion that we use in the proof of Lemma \ref{lem17}. 
\begin{lemma}\label{remlocadm}
Let $\mu\in X_*(T)$ be dominant and let $R$ be the complete local ring of $Kz^{\mu}K$ in some point $b$. Let $\mathfrak{m}$ be its maximal ideal, let $n>0$ and let $y\in LG(R)$ be the universal element over $\Spec R/\mathfrak{m}^n$. For $d\in\mathbb{N}$ we denote by $\hat K_d$ the completion of $K_d$ in $1$. Then there is a $d$ depending on $\mu$ and $n$ such that every element of $y\hat K_d$ is $\sigma$-conjugate via $\hat K_0$ to $y$.
\end{lemma}

\begin{proof}
Let $c$ be such that $q^c>n$. Let $d_0\geq \langle \mu,\alpha\rangle$ for all roots $\alpha$. As $y\in (Kz^{\mu}K)(R/\mathfrak{m}^n)$, we have $y\hat K_l\subseteq \hat K_{l-d_0} y$ for all $l\geq d_0$. Let now $d=cd_0$ and $g\in \hat K_d$. Then $g\equiv 1\pmod{\mathfrak {m}_{\hat K}}$ where $\mathfrak {m}_{\hat K}$ denotes the maximal ideal of $\mathcal{O}_{\hat K}$. Then $yg=(ygy^{-1})y$ is $\hat K_{d-d_0}$-$\sigma$-conjugate to $y\sigma(ygy^{-1})$ and $\sigma(ygy^{-1})\equiv 1\pmod{\sigma(\mathfrak {m}_K)={\mathfrak {m}_K}^q}$. Iterating this $c$ times we obtain the conclusion. 
\end{proof}

\subsection{Proof of Theorem \ref{thm1'}}

In this subsection we make the same assumptions as in Subsection \ref{subsec4.2}, in particular $G=\GL_n$.

\begin{proof}
Let $\underline{\mathcal{G}}=(K_S,b\sigma)$. In the course of the proof we will make several replacements of $(S,U_j,\underline{\mathcal{G}})$ by triples $(\tilde S,\tilde U,\tilde{\underline{\mathcal{G}}})$ of one of the following two types: In both cases, $\tilde S$ is an integral normal affine $S$-scheme of finite type, $\tilde{\underline{\mathcal{G}}}$ is the base change of $\underline{\mathcal{G}}$ to $\tilde{S}$ and 
\begin{enumerate}
\item $(\tilde{S},\tilde{U})$ is the normalization of $(S,U_j)$ in some finite field extension of the function field $\kappa(S)$ or
\item $\tilde U=U_j\times_S \tilde S$ is an open subscheme of $\tilde S$ and $\tilde U\rightarrow U_j$ is an isomorphism.
\end{enumerate}
By \cite{Vasiu}, Section 6.1.1, in both cases the map $U_j\rightarrow S$ is affine if and only if $\tilde U\rightarrow \tilde{S}$ is affine. Furthermore, $\tilde S$ is again excellent.

By \cite{HV1}, Lemma 3.13 there exists a $\mu$ such that $\underline{\mathcal{G}}$ is bounded by $\mu$. Let $l=l(\mu)$ be as in Proposition \ref{corslopefilt}. Let $\overline{\kappa}$ be an algebraic closure of $\kappa=\kappa(S)$. We consider the set of homomorphisms $g_{\overline{\kappa},d}:\overline{\kappa}[[z]]\rightarrow \mathcal{G}_{\overline{\kappa}}=\overline{\kappa}[[z]]^n$ such that the matrix coefficients of $g_{\overline{\kappa},d}$ generate an ideal containing $(z^l)$ and $b\sigma(g_{\overline{\kappa},d})\in K_{\GL_{n},d}(\overline{\kappa})g_{\overline{\kappa},d}K_{\GL_1,d}(\overline{\kappa})$. Let $c$ be as in Corollary \ref{coradmissible}. Then for each such homomorphism there is an element of $K_{\GL_{n},d-c}(\overline{\kappa})g_{\overline{\kappa},d}$ which induces a homomorphism of local shtukas $g_{\overline{\kappa}}:(\overline{\kappa}[[z]], 1\cdot\sigma)\rightarrow \underline{\mathcal{G}}_{\overline{\kappa}}$. Recall that exactly one of the Newton slopes of $\underline{\mathcal{G}}$ is equal to $0$. Hence the image of the induced morphism $\overline{\kappa}\dl z\dr\rightarrow\overline{\kappa}\dl z\dr ^n$ is the unique one-dimensional sub-vector space of $\overline{\kappa}\dl z\dr ^n$ where the slope is $0$. In particular, any two double cosets $K_{\GL_{n},d-c}(\overline{\kappa})g_{\overline{\kappa},d}K_{\GL_1,d-c}(\overline{\kappa})$ differ by an element of Aut$(\overline{\kappa}\dl z\dr, 1\cdot\sigma)\cong\mathbb{F}_q\dl z\dr$. Hence the set of double cosets satisfying the additional boundedness by $l$ is finite. We fix one double coset of some $g_{\overline{\kappa},d}$ which we choose in such a way that the ideal generated by the matrix coefficients of any representative is equal to $(z^l)$. It contains an element $\psi$ defined over a finite extension $\kappa'$ of $\kappa$. By using a replacement of $(S,U_j)$ of type (1) above, we may assume that $\psi$ is already defined over $\kappa$.

Let $V$ be a local ring of $U_j$ that is a discrete valuation ring. Let $d\in\mathbb{N}$. By our choice of $l$ there is a finite \'etale extension $V'$ of $V$ and a $g\in L\GL_n(V')$ bounded by some $\mu'=(\mu'_1,\dotsc,\mu'_{n})$ with $0\leq \mu'_i\leq l$ for all $i$ and such that $b'=g^{-1}b\sigma(g)\in LP(V')$ has only integral coordinates and that the image $b'_0$ in  the quotient $L\mathbb{G}_m(V')$ of $(L\GL_{n-1}\times L\mathbb{G}_m)(V')\cong LM(V')$ satisfies $b'_0\equiv 1\pmod{z^dV'[[z]]}$. Let $g':V'[[z]]\rightarrow V'[[z]]^n$ be the last column of $g$. Then the matrix coefficients of $g'$ are integral and generate an ideal containing $(z^l)$ and $g'b_0'= b\sigma(g')$, hence $b\sigma(g')\in K_{\GL_{n},d}(V')g'K_{\GL_1,d}(V')$. By composing $g'\in V'[[z]]^n$ with multiplication by a suitable element of $\mathbb{F}_q[[z]]$ we may assume that the class of $g'_{\overline\kappa}$ in $\overline\kappa[[z]]^n/z^{d-c}\overline\kappa[[z]]^n$ contains $\psi$. Recall that $\psi$ is defined over $\kappa$. As $V'\cap \kappa=V$ we obtain that the coset of $\psi$ contains an element $\psi_{d-c}\in V[[z]]^n/z^{d-c}V[[z]]^n$. As this holds for every local ring of $U_j$ that is a discrete valuation ring, we obtain $\psi_{d-c}\in \mathcal{O}_{U_j}[[z]]^n/z^{d-c}\mathcal{O}_{U_j}[[z]]^n$ by \cite{Matsumura}, Theorem 11.5. Now we apply \cite{Vasiu}, Lemma 6.1.5. Although this lemma is formulated for truncated Witt rings and not in our context, the proof only uses general algebraic geometry, and hence still works here. The lemma implies that after a replacement of type (2) above the morphism $\psi_{d-c}$ extends to a corresponding morphism $\psi_{d-c}:\mathcal{O}_S[[z]]/(z^{d-c})\rightarrow \mathcal{O}_S[[z]]^n/(z^{d-c})\mathcal{O}_S[[z]]^n$ over all of $S$. Let $\tilde c$ be a constant only depending on $\mu$ with $b(z^{d-c}\mathcal{O}_S[[z]]^n)\subseteq z^{d-c-\tilde c}\mathcal{O}_S[[z]]^n$ and $b^{-1}(z^{d-c}\mathcal{O}_S[[z]]^n)\subseteq z^{d-c-\tilde c}\mathcal{O}_S[[z]]^n$. Then $\psi_{d-c}$ satisfies $b\sigma(\psi_{d-c})\equiv\psi_{d-c}\pmod{z^{d-c-\tilde c}}$ because it satisfies a similar congruence over $\kappa$.

Repeating the same argument using the (bottom row of the) element $z^lg^{-1}$ instead of (the last column of) $g$, using maps $\psi':\overline\kappa[[z]]^n\rightarrow \overline\kappa[[z]]$ and replacing $d$ by $l+d$ we obtain a morphism $\psi'_{d-c}:\mathcal{O}_S[[z]]^n/(z^{d-c})\mathcal{O}_S[[z]]^n\rightarrow \mathcal{O}_S[[z]]/(z^{d-c})$ over all of $S$ such that the matrix coefficients of $\psi'_{d-c}$ generically generate an ideal containing $(z^l)$ and such that $\psi'_{d-c}b\equiv\sigma(\psi'_{d-c})\pmod{z^{d-c-\tilde c}}$. For this last congruence one also has to make use of the second assertion of Proposition \ref{corslopefilt}. The composition of $\psi'_{d-c}$ and $\psi_{d-c}$ is thus a morphism $$\psi'_{d-c}\circ \psi_{d-c}:\mathcal{O}_S[[z]]/(z^{d-c})\rightarrow \mathcal{O}_S[[z]]/(z^{d-c})$$ with $\sigma(\psi'_{d-c}\circ \psi_{d-c})\equiv \psi'_{d-c}\circ \psi_{d-c}\pmod{z^{d-c-2\tilde c}}$. The matrix of this composition is an element of $\mathcal{O}_S[[z]]/(z^{d-c})$ which generically has $z$-valuation at most $2l$. Corollary \ref{coradmissible} implies that there is a $c'$ only depending on $2l$ and hence on $\mu$, but not on $d$, such that its $K_{\GL_1,d-c-2\tilde c-c'}$-double coset contains an auto-isogeny of $(\kappa(S)[[z]], 1\cdot\sigma)=(\mathbb{F}_q[[z]],1\cdot\sigma)_{\kappa(S)}$ (assuming that we choose $d>c+2\tilde c+c'+2l$). As this means that it commutes with $\sigma$, it is defined over $\mathbb{F}_q$. The $z$-valuation of $\psi'_{d-c}\circ \psi_{d-c}$ is equal to that of the associated morphism $\mathcal{O}_S[[z]]/(z^{d-c-2\tilde c-c'})\rightarrow \mathcal{O}_S[[z]]/(z^{d-c-2\tilde c-c'})$ which is defined over $\mathbb{F}_q$. Hence the valuation is constant on all of $S$, and in particular at most $2l$ in each point. This implies similar boundedness statements (for the bound $2l$) for the ideals generated by the matrix coefficients of $\psi'_{d-c}$ and of $\psi_{d-c}$ individually. Let $x$ be a geometric point of $S$. Then by Corollary \ref{coradmissible} and Remark \ref{remcoradmissible} the double cosets of $(\psi'_{d-c})_x$ and $(\psi_{d-c})_x$ of level $d-c-2\tilde c-c'$ contain morphisms $\psi_x:(\mathcal{O}_S[[z]],1\cdot\sigma)_x\rightarrow \underline{\mathcal{G}}_x$ and $\psi'_x:\underline{\mathcal{G}}_x\rightarrow(\mathcal{O}_S[[z]],1\cdot\sigma)_x $ of local shtukas over $x$. As we choose $d>c+c'+2\tilde c+2l$, the description of the ideal generated by the matrix coefficients of $\psi'_{d-c}$ implies that $\psi'_x\neq 0$. As $(\mathcal{O}_S[[z]],1\cdot\sigma)_x$ has Newton slope $0$, at least one of the Newton slopes of $\underline{\mathcal{G}}_x$ is equal to 0. As the slope $0$ occurs precisely once in the Hodge polygon of $\underline{\mathcal{G}}_x$, the same has to hold for its Newton polygon. Hence after all replacements of types (1) and (2) made in the proof we are now in the situation where $U_j=S$ is an affine $S$-scheme, which is what we wanted to show.
\end{proof}

\subsection{Applications}

\begin{cor}\label{corthm1}
Let $S$ be an integral locally noetherian scheme and let $\underline{\mathcal{G}}=(\mathcal{G},\phi)$ be a local $G$-shtuka on $S$. Let $j\in J(\nu)$ where $\nu$ is the generic Newton point of $\underline{\mathcal{G}}$. Let $U_j$ be the open subscheme of $S$ defined by the condition that a point $x$ of $S$ lies in $U_j$ if and only if $\pr_{(j)}(\nu_{\underline{\mathcal{G}}}(x))=\pr_{(j)}(\nu)$. Then $U_j$ is an affine $S$-scheme.
\end{cor}

\begin{proof}
As in the proof of Theorem \ref{thm1} from Theorem \ref{thm1'} we see that it is enough to consider the case $G=\GL_n$. By \cite{HV1}, Proposition 2.3 a $K$-torsor is in that case locally trivial for the Zariski topology. As the statement of the corollary is also local, we may assume that $\mathcal{G}$ is trivial. Then $\phi$ corresponds to an element of $LG(S)$ and the corollary follows from Theorem \ref{thm1}. 
\end{proof}

The next corollary is the analog of de Jong--Oort's purity theorem for $p$-divisible groups. Its proof is the same as for example in \cite{Vasiu}, Remark 6.3(a) or \cite{HV1}, Corollary 7.5.
\begin{cor}\label{corpuredim}
Let $S$ be an integral, locally noetherian scheme of dimension $d$ and let $\underline{\mathcal{G}}$ be a local $G$-shtuka on $S$. Let $\nu$ be the generic Newton point and let $j\in J(\nu)$. Let $S_j\subset S$ be the complement of the open subscheme $U_j$ where $\pr_{(j)}(\nu_{\underline{\mathcal{G}}}(x))=\pr_{(j)}(\nu)$. Then $S_j$ is empty or pure of codimension $1$.
\end{cor}
\begin{proof}
By replacing $S$ by a suitable open subscheme, we may assume that $S$ is affine. Furthermore we may assume that $S_j\neq\emptyset$. By Theorem \ref{thm1}, $S\setminus S_j=U_j$ is affine. By EGA IV, Corollaire 21.12.7 this implies that $\codim(S_j)\leq 1$.
\end{proof}

We also obtain once more the usual purity for the Newton stratification associated with a local $G$-shtuka.
\begin{cor}[\cite{HV1}, Theorem 7.4]\label{cornewtonpure}
Let $S$ be an integral and locally noetherian $\mathbb{F}_q$-scheme and let $\underline{\mathcal{G}}$ be a local $G$-shtuka on $S$. Let $\nu$ be the Newton point in the generic point of $S$. Then the Newton stratification on $S$ defined by $\underline{\mathcal{G}}$ satisfies the purity property, that is the inclusion of the stratum $\mathcal{N}_{\nu}$ in $S$ is an affine morphism.
\end{cor}
\begin{proof}
For $j\in J(\nu)$ let $U_j$ be as above. Then $\mathcal{N}_{\nu}$ is the intersection of all $U_j$ for $j\in J(\nu)$. Hence by Corollary \ref{corthm1} the inclusion $\mathcal{N}_{\nu}\hookrightarrow S$ is the composition of finitely many affine morphisms and thus affine.
\end{proof}

\section{The Grothendieck conjecture}\label{secgrothconj}

We prove Theorem \ref{thm2} by reducing it to the following result.
\begin{thm}\label{thm2'}
Let $\mu_1\preceq\mu_2\in X_*(T)$ be dominant coweights. Let $S_{\mu_1,\mu_2}=\bigcup_{\mu_1\preceq\mu'\preceq \mu_2} Kz^{\mu'}K$. Let $b\in S_{\mu_1,\mu_2}(k)$ and let $R$ be the complete local ring of $S_{\mu_1,\mu_2}$ in $b$. For each $\sigma$-conjugacy class $[b']$ with $\nu_{b}\preceq\nu_{b'}\preceq \mu_2$ and $\kappa_G(b')=\kappa_G(b)$ let $\mathcal{N}_{\nu_{b'}}$ be the subscheme of $\Spec R$ where the $\sigma$-conjugacy class of the corresponding point of $S_{\mu_1,\mu_2}$ is equal to $[b']$. Then $\mathcal{N}_{\nu_{b'}}$ is non-empty and pure of codimension $\langle \rho, \mu_2-\nu_{b'}\rangle+\frac{1}{2}\defect(b')$ in $\Spec R$.
\end{thm}

For the following proof we use the dimension formula for affine Deligne-Lusztig varieties, so we briefly recall their definition. Let $b\in LG(k)$ and $\mu\in X_*(T)_{\dom}$. Then the affine Deligne-Lusztig variety associated with $b$ and $\mu$ is the locally closed subscheme $X_{\mu}(b)$ of $\Gr$ with $$X_{\mu}(b)(k)=\{g\in \Gr(k)\mid g^{-1}b\sigma(g)\in Kz^{\mu}K\}.$$ The union $X_{\preceq\mu}(b)=\bigcup_{\mu'\preceq\mu}X_{\mu'}(b)$ is a closed subscheme of $\Gr$ called the closed affine Deligne-Lusztig variety assoicated with $b$ and $\mu$.

\begin{proof}[Proof of Theorem \ref{thm2} using Theorem \ref{thm2'}]
Theorem \ref{thm2'} implies the statement about the closure of $\mathcal{N}_{\nu}$ in Theorem \ref{thm2}. Let us show that Theorem \ref{thm2'} together with Corollary \ref{cormazur} (which excludes the emptiness and thus leads to elements $b$ as in Theorem \ref{thm2'}) also implies the statement about codimensions in Theorem \ref{thm2}. Let $d>0$ be such that the Newton point of each element of $S_{\mu_1,\mu_2}$ is determined by its image in $S_{\mu_1,\mu_2}/K_d$. Then by definition the codimension in Theorem \ref{thm2} is the codimension of the image of $\mathcal{N}_{[b]}$ in $S_{\mu_1,\mu_2}/K_d$. A similar definition is used for $\mathcal{N}_{\nu_{b}}$ in $\Spec R$. Note that $S_{\mu_1,\mu_2}/K_d$ is an irreducible scheme of finite type over $k$, in particular $\dim S_{\mu_1,\mu_2}/K_d=\dim C/K_d+\codim\, C/K_d$ for every irreducible component $C$ of $\mathcal{N}_{[b]}$, see [EGA IV$_2$], Proposition 5.2.1. An analogous formula holds for the Newton strata in the completion of $S_{\mu_1,\mu_2}/K_d$ in $b$. Indeed, the completion of $S_{\mu_1,\mu_2}/K_d$ is equidimensional and catenary by \cite{HIO}, Theorem 18.17 and it is equicodimensional as it has only one closed point. As $S_{\mu_1,\mu_2}/K_d$ is noetherian, dimensions stay invariant under completion. Hence Theorem \ref{thm2'} together with Corollary \ref{cormazur} implies the statement about codimensions in Theorem \ref{thm2}. 

It remains to show that Theorem \ref{thm2'} implies Corollary \ref{cormazur}. We first consider the case that $b$ is basic. Let $w_1$ be the element of $\Omega$ whose image in $\pi_1(G)$ is equal to $\kappa_G(b)$ (cf. Section \ref{secnot}). Then $[b]$ contains an element $b_1\in LG(\mathbb{F}_q)$ which is a representative of $w_1$. Indeed, let $m>0$ such that $w_1^m$ is a translation element. Then $\ell(w_1^m)=0$, hence $w_1^m$ is central. As the Newton point of $b_1$ is equal to $\frac{1}{m}w_1^m\in X_*(T)_{\mathbb{Q}}$, we obtain that $b_1$ is basic. As $\kappa_G(b_1)=\kappa_G(b)$, we have $b_1\in [b]$. As $w_1$ is of length $0$, we have $w_1\leq z^{\mu}$ in the Bruhat order and hence $b_1\in \bigcup_{\mu'\preceq\mu}Kz^{\mu'}K$. In other words, $X_{\preceq\mu}(b)\neq\emptyset$. By \cite{GHKR}, 5.6 and 5.8 and \cite{dimdlv} we have $\dim X_{\preceq\mu}(b)\leq \langle \rho, \mu-\nu_b\rangle+\frac{1}{2}\defect(b)$ for every $\mu$ and every (not necessarily basic) $b$. Note that although \cite{GHKR} already uses Corollary \ref{cormazur} in some occasions, their proof of the upper bound on the dimension does not need this non-emptiness statement. By \cite{HV1} we know that each irreducible component of $X_{\preceq\mu}(b)$ has dimension at least $\langle \rho, \mu-\nu_b\rangle+\frac{1}{2}\defect(b)$. Thus for every pair $\mu,b$ we have $X_{\preceq\mu}(b)=\emptyset$ or $X_{\preceq\mu}(b)$ is equidimensional of that dimension and $X_{\mu}(b)$ is dense in $X_{\preceq\mu}(b)$ (compare \cite{HV1}, Theorem 1.2 and Corollary 1.3). In the basic case we already saw that $X_{\preceq\mu}(b)\neq\emptyset$. Hence $X_{\mu}(b)\neq\emptyset$ which proves the corollary for this case.

We now consider general elements $b$. Let $b_0\in G(\overline{\mathbb{F}}_q)$ be such that $b_0$ is basic and with $\kappa_G(b_0)=\kappa_G(b)$. Let $\nu'$ be its Newton point. Then $\nu'\preceq\nu\preceq\mu$, hence by what we saw above $[b_0]\cap Kz^{\mu}K\neq\emptyset$. We assume that $b_0$ is already in this intersection. Let $R$ be the complete local ring of $S_{\mu,\mu}$ at $b_0$. By Theorem \ref{thm2'} there is a point $b_1$ of $\Spec R\hookrightarrow S_{\mu,\mu}$ which is also in the $\sigma$-conjugacy class of $b$ as claimed.
\end{proof}

\begin{lemma}\label{lem17}
Let $\underline{\mathcal{G}}=(K_k,b\sigma)$ be a local $G$-shtuka over an algebraically closed field $k$. Let $\mu\in X_*(T)_{\dom}$ be such that $\underline{\mathcal{G}}$ is bounded by $\mu$. Let $R$ be the universal deformation ring for deformations of $\underline{\mathcal{G}}$ that are bounded by $\mu$. For each Newton point $\nu$ let $\mathcal{N}_{\nu}$ denote the corresponding Newton stratum in $\Spec R$. Let $R'$ be the complete local ring of $\bigcup_{\mu'\preceq\mu}Kz^{\mu'}K$ in the point $b$. For a Newton point $\nu$ let $\mathcal{N}_{\nu}'$ be the corresponding Newton stratum in $\Spec R'$ associated with the universal element in $\Spec R'\rightarrow LG$. Let $\hat K$ be the completion of $K$ at $1$.
\begin{enumerate}
\item There is an isomorphism $s':\Spec R\hat\otimes \mathcal{O}_{\hat K}\rightarrow \Spec R'$ such that image under $s'$ of each fiber of the projection $\Spec R\,\hat\otimes\, \mathcal{O}_{\hat K}\rightarrow\Spec R$ lies within one Newton stratum.
\item A point $x\in\mathcal{N}_{\nu}$ is contained in the closure of some $\mathcal{N}_{\nu'}$ if and only if $s'(x)\in\mathcal{N}_{\nu}'$ is contained in the closure of $\mathcal{N}_{\nu'}'$. 
\item The codimension of $\mathcal{N}_{\nu}'$ in $\Spec R'$ is greater than or equal to the codimension of $\mathcal{N}_{\nu}$ in $\Spec R$.
\end{enumerate}
\end{lemma}
\begin{proof}
By trivializing the universal local $G$-shtuka over $\Spec R$ we obtain a morphism $s:\Spec R\rightarrow \Spec R'$. Let $s':\Spec (R\hat\otimes \mathcal{O}_{\hat K})\rightarrow \Spec R'$ with $(x,g)\mapsto g^{-1}s(x)\sigma(g)$. To prove that $s'$ is an isomorphism we construct an inverse. Let $y\in LG(R')$ be the universal element and let $\mathfrak{m}_{R'}$ be the maximal ideal of $R'$. For all $n$ the element $y$ defines compatible deformations over $\Spec R'/\mathfrak{m}_{R'}^{q^n}$ of $\underline{\mathcal{G}}$ which are bounded by $\mu$. For every $d$ we denote by $\hat K_d$ the completion of $K_d$ in $1$. By Lemma \ref{remlocadm} each deformation of $y$ over some $\Spec R'/\mathfrak{m}_{R'}^{q^n}$ is constant (up to isomorphism) on $\hat K_{d(n)}$-cosets for some $d(n)$. Note that $\Spec (R'/\mathfrak{m}_{R'}^{q^n})/\hat K_{d(n)}$ is artinian. Hence the deformation induces a uniquely defined morphism $t_n:\Spec (R'/\mathfrak{m}_{R'}^{q^n})\rightarrow \Spec (R'/\mathfrak{m}_{R'}^{q^n})/\hat K_{d(n)}\rightarrow \Spec R$ to the universal deformation. In particular it is independent of the choice of $d(n)$, provided it is large enough. As the morphisms $t_n$ are compatible for varying $n$ they induce a morphism $t:\Spec R'\rightarrow \Spec R$. It satisfies that for all $n$ the element $y\in LG(R'/\mathfrak{m}_{R'}^{q^n})$ is $\sigma$-conjugate by an element $f_n$ of $\hat K(R'/\mathfrak{m}_{R'}^{q^n})$ to $s(t(y))\pmod{\mathfrak{m}_{R'}^{q^n}}$. To show that the elements $f_n$ are uniquely defined, that in the limit they define an element of $\hat K(R')$ and that this together with $t$ yields an inverse to $s'$ one uses the following explicit description of $f_n$: We define inductively $g_0=1$ and $g_{n+1}=s(t(y))\sigma(g_n)y^{-1}\in LG(R')$. By induction, $s(t(y))\equiv y\pmod{\mathfrak{m}_{R'}}$ implies that $y\equiv g_n^{-1}s(t(y))\sigma(g_n)\pmod{\mathfrak{m}_{R'}^{q^n}}$. Furthermore 
\begin{equation}\label{eqerg}
g_{n+1}\equiv g_ng_n^{-1}s(t(y))\sigma(g_n)y^{-1}\equiv g_n\pmod{\mathfrak{m}_{R'}^{q^n}}
\end{equation}
for all $n$. Inductively, one sees that $g_n\in LG(R'/\mathfrak{m}_{R'}^{q^n})$ is the unique element with $g_n\equiv 1\pmod{\mathfrak{m}_{R'}}$ and $y\equiv g_n^{-1}s(t(y))\sigma(g_n)\pmod{\mathfrak{m}_{R'}^{q^n}}$. In particular, $g_n=f_n$, hence $g_n\in\hat K(R'/\mathfrak{m}_{R'}^{q^n})$. By \eqref{eqerg} we see that the limit of the $g_n$ exists and thus defines an element $g\in \hat K(R')$. Together with $t$ this defines an inverse of $s'$.

By definition, the $K$-$\sigma$-conjugacy class of $g^{-1}s(x)\sigma(g)$, and thus in particular the Newton point, is constant on the fibers of the projection $\Spec R\hat\otimes \mathcal{O}_{\hat K}\rightarrow \Spec R$. Thus we proved (1). The second assertion follows directly from the above together with the property that there is a $d>0$ such that the Newton point of an element of $\bigcup_{\mu'\preceq\mu}Kz^{\mu'}K$ is determined by its image in $LG/K_d$, a consequence of the Admissibility Theorem. For the third assertion let $d$ be as above and let $\hat K_d$ be the completion of $K_d$ in $1$. Let $x$ be the generic point of some irreducible component of $\mathcal{N}'_{\nu}/\hat K_d$, and let $\tilde x$ be a lift to $LG$. Let $(t(\tilde{x}), g)$ be the image under the inverse of $s'$. If the irreducible component of $\mathcal{N}_{\nu}$ containing $t(\tilde{x})$ has codimension $c$ in $\Spec R$, there is a chain of generizations $y_0=t(\tilde{x}), y_1,\dotsc, y_c$ of $y_0$ with the following property: If $\nu_i$ denotes the Newton point of the universal local $G$-shtuka in $y_i$ then $\nu_0\preceq\dotsm\preceq\nu_i$ with $\nu_j\neq\nu_{j+1}$ for every $j$ (compare \cite{HV1}, Corollary 7.7). We consider the image in $(\Spec R')/\hat K_d$ of the chain of points $y_i$ under the composition of the morphism $\Spec R\rightarrow \Spec (R\hat\otimes \mathcal{O}_{\hat K})$ with $y\mapsto (y,g)$ for $g$ as above, the morphism $s'$, and the projection $\Spec R'\rightarrow (\Spec R')/\hat K_d$. We obtain a chain of $c$ generizations of $x$. They are contained in different Newton strata, thus in particular pairwise distinct. Hence the irreducible component of $x$ in $\mathcal{N}'_{\nu}$ has codimension at least $c$.
\end{proof}

\begin{proof}[Proof of Theorem \ref{thm2'}]
Let $R$ be as in the theorem. For $\nu_b\preceq\nu\preceq\mu_2$ let $\mathcal{N}_{\nu}\subseteq \Spec R$ be the Newton stratum of $\nu$ and let $\mathcal{N}_{\preceq\nu}=\bigcup_{\nu'\preceq\nu}\mathcal{N}_{\nu'}$. We have to show the following claim.

\noindent{\it Claim.} For all $\nu$ as above, $\mathcal{N}_{\preceq\nu}$ is pure of codimension $\delta(\nu,\mu_2)$ in $\Spec R$ and $\mathcal{N}_{\nu}$ is dense in $\mathcal{N}_{\preceq\nu}$. In particular, $\mathcal{N}_{\nu}$ is non-empty.

To prove the claim we first prove that $\codim~\mathcal{N}_{\preceq\nu}\geq\delta(\nu,\mu_2)$ for all $\nu$. We already know by \cite{HV2}, Corollary 6.10 (a) and Lemma \ref{lem17} (3) that this is true for $\nu=\nu_b$. Assume that there is a $\nu$ for which the inequality does not hold. We may in addition choose $\nu$ to be a minimal element with respect to $\preceq$ among all Newton points with that property. Let $C$ be an irreducible component of $\mathcal{N}_{\preceq\nu}$ of minimal codimension. Then by the minimality of $\nu$ the generic Newton point on $C$ is equal to $\nu$. Let $C'$ be an irreducible component of the complement of $\mathcal{N}_{\nu}$ in $C$ (which is nonempty as $C$ contains the special point of $\Spec R$ and $\nu\neq\nu_b$). By Theorem \ref{thm1} (or already by the usual purity for the Newton stratification as in Corollary \ref{cornewtonpure}) the codimension of $C'$ is equal to $\codim~C+1$. Let $\nu'$ be the generic Newton point on $C'$. Then $\nu'\preceq\nu$ with $\nu'\neq\nu$. Hence $\delta(\nu',\mu_2)\geq\delta(\nu,\mu_2)+1>\codim~C+1=\codim~C'$. This is a contradiction to the minimality of $\nu.$

It remains to show that for all $\nu$ the codimension of each irreducible component of $\mathcal{N}_{\preceq\nu}$ is at most $\delta(\nu,\mu_2)$ and that $\mathcal{N}_{\nu}$ is dense in $\mathcal{N}_{\preceq\nu}$. By \cite{HV2}, Corollary 6.10 (b)  and Lemma \ref{lem17} (2) this is true for $\nu=\mu_2$. We use decreasing induction (for the partial ordering $\preceq$). By Remark \ref{remchai} (\ref{remchai2}) it is enough to show that the above statement for $\nu$ implies the corresponding statement for $\nu'$ whenever $\nu_b\preceq\nu'\preceq\nu$ with $\delta(\nu',\nu)=1$, hence $\delta(\nu',\mu_2)=\delta(\nu,\mu_2)+1$. Let $j\in J(\nu)$ with $\pr_{(j)}(\nu)\neq\pr_{(j)}(\nu')$, compare Remark \ref{remnew}. Then $\bigcup_{\eta\preceq \nu, \pr_{(j)}(\eta)\neq\pr_{(j)}(\nu)}\mathcal{N}_{\eta}=\mathcal{N}_{\preceq\nu'}$. Furthermore, this subscheme is nonempty as it contains the closed point of $\Spec R$. By Corollary \ref{corpuredim}, $\mathcal{N}_{\preceq\nu'}$ is pure of codimension 1 in $\mathcal{N}_{\preceq\nu}$. Especially, the codimension in $\Spec R$ of each irreducible component is $\delta(\nu',\mu_2)=\delta(\nu,\mu_2)+1$. We already know that for all $\eta\preceq\nu'$ with $\eta\neq\nu'$ we have $\codim~\mathcal{N}_{\preceq\eta}\geq\delta(\eta,\mu_2)>\delta(\nu',\mu_2)$. Hence $\mathcal{N}_{\nu'}$ is dense in $\mathcal{N}_{\preceq\nu'}$.
\end{proof}

\end{document}